\documentclass[12pt,a4paper]{article}
\usepackage{color}
\usepackage{xcolor}
\usepackage{graphicx}
\usepackage{tabularx}
\usepackage{amsmath,amssymb,mathrsfs,wasysym}
\usepackage{amsthm}
\usepackage{multicol}
\usepackage{enumerate}
\usepackage{float}
\usepackage{makeidx}  
\date{}
\usepackage{authblk}

\newtheorem{theorem}{Theorem}
\newtheorem{lem}{Lemma}

\newtheorem{cor}{Corollary}
\theoremstyle{definition}

\newtheorem{ex}{Example}
\newtheorem{rem}{Remark}

\begin{document}
\title{\bf Orbits of Free Cyclic Submodules over
Rings of Lower Triangular Matrices
}
\author{\bf Edyta Bartnicka}
\affil{\footnotesize Warsaw University of Life Sciences, Faculty of Applied Informatics and Mathematics\\ Warsaw, Poland\\
{\tt edytabartnicka@wp.pl}}
\maketitle
\begin{abstract}
\noindent
Given a ring $T_n\ (n\geqslant 2)$ of lower triangular $n\times n$ matrices with entries from an arbitrary field $F$, we completely classify the orbits of free cyclic submodules of $^2T_n$ under the action of the general linear group $GL_2(T_n)$. Interestingly, the total number of such orbits is found to be equal to the Bell number $B_n$. A representative of each orbit is also given.
\end{abstract}


\section{Introduction}
The general linear group $GL_2(R)$ acts in natural way on the set of all free cyclic submodules of  the two-dimensional left free module ${^2R}$. In the case of an arbitrary associative ring $R$ with unity,  $GL_2(R)$-orbit of free cyclic submodules generated by unimodular pairs is the projective line $\mathbb{P}(R)$, which has already been discussed in a number of papers (see, e.g., \cite{hav1, graph, planes}). We also refer to \cite{her, design} for the definition of the unimodularity. There also exist  rings features non-unimodular free cyclic submodules.
In this note we investigate the case of a ring ${T_n}$ of lower triangular $n\times n$ matrices $A=\left[\begin{array}{ccclr}
a_{11}&0&\ldots&0\\a_{21}&a_{22}&\ldots&0\\ \vdots& \vdots&\ddots&\vdots\\ a_{n1}&a_{n2}&\dots&a_{nn}\end{array}\right]$ over a fixed field $F$, where $n\geqslant 2$.

The main motivation for this paper is the work of Havlicek and Saniga \cite{hs}, where a thorough classification of the vectors of the free left module $^2T_2$ over a ring of ternions $T_2$ up to the natural action of $GL_2(T_2)$ was given; it was found that there exactly two (one unimodular and one non-unimodular) orbits of free cyclic submodules of ${^2T_2}$ under the action of $GL_2(T_2)$.

In \cite{fcs} there are described all orbits of free cyclic submodules in the ring ${T_3}$ under the action of $GL_2(T_3)$. Here we extend and generalize the findings of \cite{fcs} to an arbitrary ring ${T_n}$ and provide the full classification of orbits of free cyclic submodules of ${^2T_n}$.
The starting point of our analysis is a complete characterization of free cyclic submodules, unimodular pairs and outliers generating free cyclic submodules in the case of $T_n$.
Next, we will introduce representatives of all $GL_2(T_n)$-orbits (Theorems \ref{dist.representatives} and \ref{representatives}).
Finally, we will prove that the total number of $GL_2(T_n)$-orbits  is equal to the Bell number $B_n$ (Theorem \ref{number.bell}).


\section{Free cyclic submodules $T_n\left(A, B\right)$}
Consider the free left module ${^2R}$ over a ring $R$ with unity. Let $\left(a, b\right)\subset{^2R}$, then the set $R\left(a, b\right)=\{r\left(a, b\right);r\in R\}$ is a left cyclic submodule of ${^2R}$. If the equation $(ra, rb)=(0, 0)$ implies that $r=0$, then $R\left(a, b\right)$ is called {\it free}.
We  give now the condition for a cyclic submodule $M_n(A, B)\subset {^2M_n}$, where $M_n$ is the ring of $n\times n$ matrices over a fixed field $F$, to be free.

\begin{lem}
A cyclic submodule $M_n\left(A, B\right)$ is free if, and only if, the rank of $[A|B]$ is $n$.\label{free.matrix}
\end{lem}
\begin{proof}
Let $k_j$ be the $j$-th column of $[A|B]$ and let $w_i=\left[\begin{array}{cclr}x_{i1} x_{i2} \ldots x_{in}\end{array}\right], i=1, 2, \ldots, n$.\\
 Then
$[A|B]=[k_1 k_2\ldots k_{2n}]$ and $X=\left[\begin{array}{cclr}
x_{11}&x_{12}&\ldots&x_{1n}\\x_{21}&x_{22}&\ldots&x_{2n}\\ \vdots &\vdots&\vdots&\vdots\\ x_{n1}&x_{n2}&\ldots&x_{nn}\end{array}\right]=\left[\begin{array}{cclr}w_{1}\\w_{2}\\ \vdots\\ w_{n}\end{array}\right]$. \\
$"\Leftarrow"$ Suppose that the rank of $[A|B]$ is $n$. Equivalently, exactly $n$ columns $k_j$ of $[A|B]$ represent a set of  linearly independent vectors. Then the equation $$\left[\begin{array}{cclr}w_{1}\\w_{2}\\ \vdots\\ w_{n}\end{array}\right][k_1 k_2 \ldots k_{2n}]=0,$$ (where $0$ denotes the zero matrix of size $n \times 2n$) implies that linear functionals $w_1, w_2, \ldots, w_n:F^n\rightarrow F$  are zeros on $n$  linearly independent vectors of the space $F^n$. So they are zero and hence $M_n\left(A, B\right)$ is free. \\
 $"\Rightarrow"$ Assume that a submodule $M_n\left(A, B\right)$ is free. Then the fact that any set of linear functionals $w_1, w_2, \ldots, w_n:F^n\rightarrow F$  are zeros on vectors $k_1, k_2, \dots, k_{2n}$ implies $w_1= w_2= \ldots= w_n=0$. Consequently, there exists a set of $n$ linearly independent vectors $k_j$ and so the rank of $[A|B]$ is $n$.
\end{proof}
\begin{rem}
 The statement of Lemma \ref{free.matrix} remains true if one replaces the ring $M_n$ by any of its subrings, in particular by the ring $T_n$.\label{free.triangular}
\end{rem}

It is well known that any unimodular pair  generates a free cyclic submodule; we shall call such free cyclic submodules unimodular. Now, we will introduce the condition of unimodularity in the case of $T_n$.
\begin{rem} A pair $(A, B)\subset{^2T_n}$ is {\sl unimodular} if, and only if, $a_{ii}\neq 0 \vee b_{ii}\neq 0$ for any $i=1,2, \dots, n$.\label{unim.tri}
\end{rem}
\begin{proof}
According to the general definition of unimodularity,  a pair $(A,B)\subset{^2T_n}$ is unimodular if there exist matrices $X, Y\in T_n$ such that $AX+BY=U\in  T_n^*$, where $T_n^*$ denotes the group of invertible elements of the ring $T_n$. As
$$AX+BY=\left[\begin{array}{ccclr}
a_{11}x_{11}+b_{11}y_{11}&&0\\&\ddots &\\ *&&a_{nn}x_{22}+b_{nn}y_{nn}\end{array}\right]$$ for any $X, Y\in T_n(q)$, then
$AX+BY=U\Leftrightarrow a_{ii}x_{ii}+b_{ii}y_{ii}\neq 0$ for any $i=1,2,\ldots, n$. This is equivalent to  $a_{ii}x_{ii}\neq -b_{ii}y_{ii}$ for any $i=1,2,\ldots, n$. Hence $(A, B)\subset{^2T_n}$ is {\sl unimodular} if, and only if, $a_{ii}\neq 0 \vee b_{ii}\neq 0$ for any $i=1,2, \dots, n$.
\end{proof}

Another type of free cyclic  submodules is the one  represented by pairs that are not contained in any cyclic submodule generated by a unimodular pair, so-called {\sl outliers} (first introduced in \cite{shpp}). As it was pointed out in \cite{hs, fcs}, the number of such non-unimodular free cyclic submodules can be quite large although, of course, not all outliers generate free cyclic submodules.

\begin{lem}
A pair $\left(A, B\right)\subset{^2T_n}$ is
an outlier generating a free cyclic submodule if, and only if, $\left(A, B\right)$ is non-unimodular and $rank[A|B]=n$.\label{out.gen}
\end{lem}
\begin{proof}
According to \cite[Theorem 1(2)]{fcs} if  a non-unimodular pair $\left(A, B\right)\subset{^2T_n}$\linebreak  generates a free cyclic submodule, then $\left(A, B\right)$ is an outlier. Lemma \ref{free.matrix} and Remark \ref{free.triangular} then yield the desired claim.
\end{proof}

It also follows that there are no other free cyclic submodules  $T_n\left(A, B\right)$ apart from those generated by unimodular pairs or outliers. Moreover, a cyclic submodule  $T_n\left(A, B\right)$ generated by a unimodular pair cannot have non-unimodular generators  (see \cite[Corollary 1]{fcs}).

\section{Representatives of $GL_2(T_n)$-orbits}
The general linear group $GL_2(T_n)$ acts in a natural way (from the right) on the free left $T_n$-module. Orbits of the set of all free cyclic submodules $T_n(q)\left(A, B\right)\subset {^2T_n(q)}$ under the action of $GL_2(T_n)$ are called $GL_2(T_n)$-orbits. To show the explicit form of representatives of such  orbits, we will use the following remark:

\begin{rem}\label{invertibility}
Let $X, Y, W, Z\in T_n$. The matrix $\left[\begin{array}{cclr}
X&Y\\ W&Z\end{array}\right]$ is invertible, i.e. it is an element of the group $GL_2(T_n)$, if, and only if, $x_{ii}z_{ii}-y_{ii}w_{ii}\neq 0$ for all $i=1, 2, \ldots, n$.
\end{rem}

Let ${\tt(A, B)}$ be a pair of ${^2T_n}$ such that for all $i=1, 2, \dots, n, j=1, 2, \dots, i-1$:
\begin{itemize}
\item $a_{ii}, b_{ij}\in \{0, 1\}$,
\item $a_{ij}=b_{ii}=0$,
\item the number of non-zero entries of ${\tt [A|B]}$ is equal to its rank and it is $n$. \end{itemize}
We will use the grotesk typeface for such pairs. According to Lemma \ref{free.matrix} and Remark \ref{free.triangular} any pair ${\tt(A, B)}\in{^2T_n}$ generates a free cyclic submodule.

\begin{theorem}\label{dist.representatives}
Let ${\tt(A, B)}, {\tt(C, D)}\in{^2T_n}$. Submodules $T_n{\tt(A, B)}, T_n{\tt(C, D)}$  are representatives of distinct $GL_2(T_n)$-orbits if, and only if, ${\tt(A, B)}\neq{\tt(C, D)}$.
\end{theorem}
\begin{proof}
$"\Rightarrow"$ This is obvious.\newline
$"\Leftarrow"$  We give the proof by contradiction.
Assume that $T_n{\tt(A, B)}, T_n{\tt(C, D)}$  are in the same $GL_2(T_n)$-orbit, i.e. there exists $\left[\begin{array}{cclr}
X&Y\\ W&Z\end{array}\right]\in GL_2(T_n)$ such that ${\tt(A, B)}\left[\begin{array}{cclr}
X&Y\\ W&Z\end{array}\right]={\tt(C, D)}$. Hence,
$$\begin{cases}
a_{ii}x_{ii}=c_{ii},\\a_{ii}y_{ii}=0,\end{cases}$$ for any $i\in \{1, 2, \dots, n\}$.

Submodules $T_n{\tt(A, B)}, T_n{\tt(C, D)}$  are free, thus $a_{11}, c_{11}\in F^*=F\backslash\{0\}$, and so $a_{11}=c_{11}=1$.

Let $i\in \{2, 3, \dots, n\}$. Consider two cases:
\begin{enumerate}
\item $a_{ii}=1$; Then by the above system of equations and according to Remark \ref{invertibility} we get $y_{ii}=0$,  $x_{ii}\neq 0$ and $c_{ii}\neq 0$. Consequently, $c_{ii}=a_{ii}=1$ and so all the remaining  entries in the $i$-th row of matrices {\tt A, B, C, D} are equal to zero.
\item $a_{ii}=0$; Then we get immediately $c_{ii}=a_{ii}=0$. The form of ${\tt(A, B)}\in{^2T_n}$ implies that there exists exactly one non-zero entry $b_{ij}$ in the $i$-th row of the matrix {\tt B}. Suppose then that $b_{ij}=1$ for some $j\in \{1, 2, \dots, i-1\}$. We obtain the following system of equations:
$$\begin{cases}c_{ij}=0=a_{ii}x_{ij}+b_{ij}w_{jj}+b_{i(j+1)}w_{(j+1)j}+ \ldots +b_{i(i-1)}w_{(i-1)j}=b_{ij}w_{jj},\\
d_{ij}=a_{ii}y_{ij}+b_{ij}z_{jj}+b_{i(j+1)}z_{(j+1)j}+ \ldots +b_{i(i-1)}z_{(i-1)j}=b_{ij}z_{jj}.\end{cases}$$
Therefore $w_{jj}=0$. Remark \ref{invertibility}  leads to $z_{jj}\neq 0$, and so $d_{ij}\neq 0$. Thus $d_{ij}=b_{ij}=1$ and  all the remaining  entries in the $i$-th row of matrices ${\tt A, B, C, D}$ are equal to zero.
\end{enumerate}
This means that ${\tt(A, B)}={\tt(C, D)}$, which completes the proof.
\end{proof}
\begin{theorem}\label{representatives}
Any $GL_2(T_n)$-orbit  has a unique representative $T_n{\tt(A, B)}\subset{^2T_n}$.
\end{theorem}
\begin{proof}
Taking into account Theorem \ref{dist.representatives} it suffices to show that any $GL_2(T_n)$-orbit has a representative of the form $T_n{\tt(A, B)}\subset{^2T_n}$. This is equivalent to saying that for any free cyclic submodule $T_n\left(A, B\right)\subset{^2T_n}$ there exist matrices $Q\in GL_2(T_n)$ and $U\in T_n^*$ such that
$U(A, B)Q={\tt(A, B)}.$
We will now show that this indeed holds.

\noindent
\item Let $T_n\left(A, B\right)$ be any free cyclic submodule of ${^2T_n}$.
We multiply $\left(A, B\right)\in {^2T_n}$ by  $\left[\begin{array}{cclr}
X&Y\\ W&Z\end{array}\right]\in GL_2(T_n)$, with entries $y_{ii}=-a_{ii}^{-1}b_{ii}z_{ii}$ if $a_{ii}\neq 0$, and
$z_{ii}=0$ if $a_{ii}= 0$, $i=1, 2, \dots, n$.  This yields
$$(A, B)\left[\begin{array}{cclr}
X&Y\\ W&Z\end{array}\right]=(C, D),$$
where $d_{ii}=0$ for any $i=1, 2, \dots, n$.

\noindent
The next step is multiplying
$$P^{-1}\left(C, D\right)\left[\begin{array}{cclr}
P&0\\ 0&I\end{array}\right]=\left(P^{-1}CP, P^{-1}D\right)=(J, E),$$
where $P\in T_n$ and $J$ is a Jordan normal form of $C$ with block matrices $\lambda_i$, which are just elements of $F$. Of course, $e_{ii}=0$ for any $i=1, 2, \dots, n$.

\noindent
Let $J'$ be the matrix obtained from $J$ by replacing any $\lambda_i\neq 0$ by $\lambda_i^{-1}$, and $\lambda_i=0$ by $1$. The result of multiplication $J'(J, E)=(J'J, J'E)$ is a pair $({\tt A}, F)$, where, for any $i=1, 2, \dots, n$, $a_{ii}=\begin{cases}1, \lambda_i\neq 0\\
0, \lambda_i=0\end{cases}$ and $f_{ii}=0$.

\noindent
We multiply again the last pair $({\tt A}, F)$ by
$\left[\begin{array}{cclr}
I&-F\\ 0&I\end{array}\right]$ and obtain a pair $({\tt A}, -{\tt A}F+F)=({\tt A}, G)$, where $g_{ij}=0$ for any $i, j=1, 2, \dots, n$ such that $a_{ii}=1$.

\noindent
Suppose that $i_t, t\in T$ are the numbers of non-zero rows of matrix $G$, where $i_t<i_s$ if $t<s$. Let now
\begin{enumerate}
\item $j_t=j_1$ be the number of column of $G$ such that $g_{i_1j_1}\neq 0$ and $g_{i_1k}= 0$ for any $k>j_1$;
\item $j_t, t>1$ be the numbers of columns of $G$ which fulfils all the requirements set:\label{2}
\begin{enumerate}[a.]
\item  $j_t\neq j_r$ for any $r=1, 2, \dots, t-1$,\label{a}
\item $g_{i_tj_t}\neq 0$,
\item $g_{i_tk}= 0$ for any $k$ satisfying the two conditions: $k>j_t$ and $k\neq j_r$ for any $r=1, 2, \dots, t-1$,
    \item a sequence of vectors $\left[\begin{array}{cclr}
g_{i_1j_1}\\g_{i_2j_1}\\\vdots\\g_{i_tj_1}\end{array}\right],
\left[\begin{array}{cclr}
g_{i_1j_2}\\g_{i_2j_2}\\\vdots\\g_{i_tj_2}\end{array}\right],
\dots, \left[\begin{array}{cclr}
g_{i_1j_t}\\g_{i_2j_t}\\\vdots\\g_{i_tj_t}\end{array}\right]$ is linearly independent (the possibility of choice of such $j_t$ is guaranteed because the rank of $G$ is $|T|$ by the Lemma \ref{free.matrix} and Remark \ref{free.triangular}).\label{d}
    \end{enumerate}
\end{enumerate}
Consider the matrix $V$ of $T_n$ such that
$$\sum_{\alpha=j_t}^{i_t-1}g_{i_t\alpha}v_{\alpha j_t}=1$$
and 
$$\sum_{\alpha=l}^{i_t-1}g_{i_t\alpha}v_{\alpha l}=0$$
for any $t\in T,\ l=1, 2, \dots, j_t-1$.

\noindent
We prove now that $V$ is invertible, i.e. we show that $v_{j_tj_t}\neq 0$ for each $j_t$ one by one, from the greatest to the smallest, by using the method described below. Let $j_{t_m}=\max\{j_t, t\in T\}$ 
and let $t$ be a fixed element of $T$. Assume that $j_{t_1}<j_{t_2}<\dots <j_{t_m}$ be all numbers of columns of $G$  greater then $j_t$ and  meet the above requirements \ref{a} - \ref{d}. From the conditions for $V$ we obtain a system of equations:
$$\begin{cases}g_{i_tj_t}v_{j_tj_t}+g_{i_tj_{t_1}}v_{j_{t_1}j_t}+\dots + g_{i_tj_{t_m}}v_{j_{t_m}j_t}=1\\
g_{i_{t_1}j_t}v_{j_tj_t}+g_{i_{t_1}j_{t_1}}v_{j_{t_1}j_t}+\dots + g_{i_{t_1}j_{t_m}}v_{j_{t_m}j_t}=0\\
g_{i_{t_2}j_t}v_{j_tj_t}+g_{i_{t_2}j_{t_1}}v_{j_{t_1}j_t}+\dots + g_{i_{t_2}j_{t_m}}v_{j_{t_m}j_t}=0\\
\vdots\\
g_{i_{t_m}j_t}v_{j_tj_t}+g_{i_{t_m}j_{t_1}}v_{j_{t_1}j_t}+\dots + g_{i_{t_m}j_{t_m}}v_{j_{t_m}j_t}=0\end{cases}$$
Obviously, $v_{j_tj_t}=g_{i_tj_t}
^{-1}$ if $j_t=j_{t_m}$, otherwise
we solve these equations from the last to the first one for $v_{j_{t_m}j_t}, v_{j_{t_{m-1}}j_t},\dots, v_{j_{t_1}j_t}$, respectively. Plugging these solutions to the first equation and making  simple calculations yields
$$v_{j_tj_t}=\left[\begin{array}{cccclr}
\left|\begin{array}{cccclr}
g_{i_tj_t}&g_{i_tj_{t_1}}&\dots&g_{i_tj_{t_m}}\\ g_{i_{t_1}j_t}&g_{i_{t_1}j_{t_1}}&\dots&g_{i_{t_1}j_{t_m}}\\
\vdots&\vdots&\vdots&\vdots\\
g_{i_{t_m}j_t}&g_{i_{t_m}j_{t_1}}&\dots&g_{i_{t_m}j_{t_m}}\\
\end{array}\right|v_{j_{t_1}j_{t_1}}v_{j_{t_2}j_{t_2}}\dots v_{j_{t_m}j_{t_m}} \end{array}\right]^{-1}.$$
We should mention that the conditions \ref{a} - \ref{d} for $j_t$ provide that the above determinant is different from zero.

The multiplication $({\tt A}, G)\left[\begin{array}{cclr}
I&0\\ 0&V\end{array}\right]$ gives a pair $({\tt A}, H)$, where for any pair of indices $i_t, j_t, t\in T$ and for any $l=1, 2, \dots, j_t-1$ we have: $h_{i_tj_t}=g_{i_tj_t}v_{j_tj_t}=g_{i_tj_t}g_{i_tj_t}^{-1}=1$ if $j_t=i_t-1$, otherwise
$$h_{i_tj_t}=g_{i_tj_t}v_{j_tj_t}+g_{i_t(j_t+1)}v_{(j_t+1)j_t}+\dots
+g_{i_t(i_t-1)}v_{(i_t-1)j_t}=$$
$$=g_{i_tj_t}g_{i_tj_t}^{-1}\left(1-g_{i_t(j_t+1)}v_{(j_t+1)j_t}-g_{i_t(j_t+2)}v_{(j_t+2)j_t}-\dots
-g_{i_t(i_t-1)}v_{(i_t-1)j_t}\right)+$$
$$+g_{i_t(j_t+1)}v_{(j_t+1)j_t}+\dots
+g_{i_t(i_t-1)}v_{(i_t-1)j_t}=1,$$ and
$$h_{i_tl}=g_{i_tl}v_{ll}+g_{i_t(l+1)}v_{(l+1)l}+\dots
+g_{i_t(i_t-1)}v_{(i_t-1)l}=$$
$$=g_{i_tl}v_{ll}+g_{i_t(l+1)}v_{(l+1)l}+\dots +g_{i_t(j_t-1)}v_{(j_t-1)l}+g_{i_tj_t}v_{j_tl}=$$
$$=g_{i_tl}v_{ll}+g_{i_t(l+1)}v_{(l+1)l}+\dots +g_{i_t(j_t-1)}v_{(j_t-1)l}+$$
$$+g_{i_tj_t}\left[-g_{i_tj_t}^{-1}\left(g_{i_tl}v_{ll}+g_{i_t(l+1)}v_{(l+1)l}+\dots
+g_{i_t(j_t-1)}v_{(j_t-1)l}\right)\right]=0$$
if $j_t=i_t-1$, otherwise
$$h_{i_tl}=g_{i_tl}v_{ll}+g_{i_t(l+1)}v_{(l+1)l}+\dots
+g_{i_t(i_t-1)}v_{(i_t-1)l}=$$
$$=g_{i_tl}v_{ll}+g_{i_t(l+1)}v_{(l+1)l}+\dots +g_{i_t(j_t-1)}v_{(j_t-1)l}+$$ $$+g_{i_tj_t}v_{j_tl}+g_{i_t(j_t+1)}v_{(j_t+1)l}+\dots
+g_{i_t(i_t-1)}v_{(i_t-1)l}=$$
$$=g_{i_tl}v_{ll}+g_{i_t(l+1)}v_{(l+1)l}+\dots +g_{i_t(j_t-1)}v_{(j_t-1)l}+$$ $$+g_{i_tj_t}
\Big\{-g_{i_tj_t}^{-1}\big[\left(g_{i_tl}v_{ll}+g_{i_t(l+1)}v_{(l+1)l}+\dots
+g_{i_t(j_t-1)}v_{(j_t-1)l}\right)+$$
$$+\left(g_{i_t(j_t+1)}v_{(j_t+1)l}+g_{i_t(j_t+2)}v_{(j_t+2)l}+\dots
+g_{i_t(i_t-1)}v_{(i_t-1)l}\right)\big]\Big\}+$$
$$+g_{i_t(j_t+1)}v_{(j_t+1)l}+g_{i_t(j_t+2)}v_{(j_t+2)l}+\dots
+g_{i_t(i_t-1)}v_{(i_t-1)l}=0.$$
Clearly, for any $i=1, 2, \dots, n$, and $j=1, 2, \dots, i-1$ we have: $h_{ii}=g_{ii}v_{ii}=0$, and  $h_{ij}=g_{ij}v_{jj}+g_{i(j+1)}v_{(j+1)j}+\dots+ g_{i(i-1)}v_{(i-1)j}=0$ if $g_{ij}=0$ for all   $j=1, 2, \dots, i-1$.

\noindent
Now, we will furnish the last part of the proof. Let $K$ be a matrix of $T_n^*$ such that $k_{ii}=1$ for all $i=1, 2, \dots, n$ and $$k_{ii'}=-k_{i(i'+1)}h_{(i'+1)j}-k_{i(i'+2)}h_{(i'+2)j}-\dots-k_{ii}h_{ij}$$ if the $j$-th column of $H$ is of the form: $h_{i'j}=1, h_{mj}=0$ for all $m<i'$ and $h_{ij}\neq 0$ for some $i>i'$. Then $K({\tt A}, H)=({\tt A}, L)$, where $$l_{ij}=k_{i(j+1)}h_{(j+1)j}+k_{i(j+2)}h_{(j+2)j}+\dots+k_{ii}h_{ij}=$$
$$=k_{ii'}+k_{i(i'+1)}h_{(i'+1)j}+k_{i(i'+2)}h_{(i'+2)j}+\dots+k_{ii}h_{ij}=0.$$ Therefore, $({\tt A}, L)=({\tt A}, \tt B)$, which completes the proof.
\end{proof}
\begin{cor}\label{number.pairs}
The total number of $GL_2(T_n)$-orbits is equal to the number of pairs ${\tt(A, B)}\in{^2T_n}$.
\end{cor}

\begin{ex}
Consider the  pair
$$({\tt A}, G)=\left(\left[\begin{array}{ccccclr}
1&0&0&0&0&0&0\\
0&1&0&0&0&0&0\\
0&0&0&0&0&0&0\\
0&0&0&1&0&0&0\\
0&0&0&0&0&0&0\\
0&0&0&0&0&0&0\\
0&0&0&0&0&0&0
\end{array}\right], \left[\begin{array}{ccccclr}
0&0&0&0&0&0&0\\
0&0&0&0&0&0&0\\
g_{31}&g_{32}&0&0&0&0&0\\
0&0&0&0&0&0&0\\
g_{51}&g_{52}&g_{53}&g_{54}&0&0&0\\
g_{61}&g_{62}&g_{63}&g_{64}&0&0&0\\
g_{71}&g_{72}&g_{73}&g_{74}&0&0&0\\
\end{array}\right]\right)\in {^2T_7}$$ where $g_{32}, g_{54}, g_{63}, g_{71}\in F^*$.
Based on the proof of Theorem \ref{representatives} we find the indices $i_t, j_t, t\in T$ and  determine the corresponding entries of $V\in T_7^*$:
\begin{itemize}
\item $i_1=3, j_1=2$,

$v_{22}=g_{32}^{-1}$,

$v_{21}=-g_{32}^{-1}g_{31}v_{11}$;

\item $i_2=5, j_2=4$,

$v_{44}=g_{54}^{-1}$,

$v_{41}=-g_{54}^{-1}\left(g_{51}v_{11}+g_{52}v_{21}+g_{53}v_{31}\right)$,

$v_{42}=-g_{54}^{-1}\left(g_{52}v_{22}+g_{53}v_{32}\right)$,

$v_{43}=-g_{54}^{-1}g_{53}v_{33}$;

\item $i_3=6, j_3=3$,

$v_{33}=g_{63}^{-1}\left(1-g_{64}v_{43}\right)$,

$v_{31}=-g_{63}^{-1}\left(g_{61}v_{11}+g_{62}v_{21}+g_{64}v_{41}\right)$,

$v_{32}=-g_{63}^{-1}\left(g_{62}v_{22}+g_{64}v_{42}\right)$;

\item $i_4=7, j_4=1$,

$v_{11}=g_{71}^{-1}\left(1-g_{72}v_{21}-g_{73}v_{31}-g_{74}v_{41}\right).$\end{itemize}
The result of multiplication $({\tt A}, G)\left[\begin{array}{cclr}
I&0\\ 0&V\end{array}\right]$ is a pair $({\tt A}, H)\in {^2T_7}$, where:

$h_{32}=g_{32}v_{22}=g_{32}g_{32}^{-1}=1,$
$h_{31}=g_{31}v_{11}+g_{32}v_{21}=g_{31}v_{11}+g_{32}\left(-g_{32}^{-1}g_{31}v_{11}\right)=$ $=0,$

$h_{54}=g_{54}v_{44}=g_{54}g_{54}^{-1}=1,$
$h_{51}=g_{51}v_{11}+g_{52}v_{21}+g_{53}v_{31}+g_{54}v_{41}=g_{51}v_{11}+$ $+g_{52}v_{21}+g_{53}v_{31}+g_{54}
\left[-g_{54}^{-1}\left(g_{51}v_{11}+g_{52}v_{21}+g_{53}v_{31}\right)\right]=0,$
$h_{52}=g_{52}v_{22}++g_{53}v_{32}+g_{54}v_{42}=g_{52}v_{22}+g_{53}v_{32}+g_{54}
\left[-g_{54}^{-1}\left(g_{52}v_{22}+g_{53}v_{32}\right)\right]=0,$

$h_{53}=g_{53}v_{33}+g_{54}v_{43}=g_{53}v_{33}+g_{54}
\left(-g_{54}^{-1}g_{53}v_{33}\right)=0,$

$h_{63}=g_{63}v_{33}+g_{64}v_{43}=g_{63}g_{63}^{-1}\left(1-g_{64}v_{43}\right)+g_{64}v_{43}=1,$
$h_{61}=g_{61}v_{11}+g_{62}v_{21}++g_{63}v_{31}+g_{64}v_{41}=g_{61}v_{11}+g_{62}v_{21}+g_{63}\left[-g_{63}^{-1}
\left(g_{61}v_{11}+g_{62}v_{21}+g_{64}v_{41}\right)\right]+$ $+g_{64}v_{41}=0,$
$h_{62}=g_{62}v_{22}+g_{63}v_{32}+g_{64}v_{42}=g_{62}v_{22}+g_{63}
\left[-g_{63}^{-1}\left(g_{62}v_{22}+g_{64}v_{42}\right)\right]+$ $+g_{64}v_{42}=0,$

$h_{71}=g_{71}v_{11}+g_{72}v_{21}+g_{73}v_{31}+g_{74}v_{41}=g_{71}g_{71}^{-1}
\left(1-g_{72}v_{21}-g_{73}v_{31}-g_{74}v_{41}\right)+$ $+g_{72}v_{21}+g_{73}v_{31}+g_{74}v_{41}=1.$

\noindent
We have  $h_{ii}=g_{ii}v_{ii}=0$ for any $i=1, 2, \dots, 7$ and  $h_{21}=g_{21}v_{11}=0$, $h_{4j}=g_{4j}v_{jj}+g_{4(j+1)}v_{(j+1)j}+\dots+g_{43}v_{3j}=0 $  for $j=1, 2, 3$. Moreover, $h_{65}=g_{65}v_{55}=0$, $h_{75}=g_{75}v_{55}+g_{76}v_{65}=0$, $h_{76}=g_{76}v_{66}=0.$ Thus
$$H=\left[\begin{array}{ccccclr}
0&0&0&0&0&0&0\\
0&0&0&0&0&0&0\\
0&1&0&0&0&0&0\\
0&0&0&0&0&0&0\\
0&0&0&1&0&0&0\\
0&0&1&h_{64}&0&0&0\\
1&h_{72}&h_{73}&h_{74}&0&0&0\\
\end{array}\right].$$

Then $K({\tt A}, H)={\tt(A, B)}
$, where
$$K=\left[\begin{array}{ccccclr}
1&0&0&0&0&0&0\\
0&1&0&0&0&0&0\\
0&0&1&0&0&0&0\\
0&0&0&1&0&0&0\\
0&0&0&0&1&0&0\\
0&0&0&0&-h_{64}&0&0\\
0&0&-h_{72}&0&h_{73}h_{64}-h_{74}&-h_{73}&1\\
\end{array}\right]$$
and $({\tt A}, \tt B)=\left(\left[\begin{array}{ccccclr}
1&0&0&0&0&0&0\\
0&1&0&0&0&0&0\\
0&0&0&0&0&0&0\\
0&0&0&1&0&0&0\\
0&0&0&0&0&0&0\\
0&0&0&0&0&0&0\\
0&0&0&0&0&0&0
\end{array}\right], \left[\begin{array}{ccccclr}
0&0&0&0&0&0&0\\
0&0&0&0&0&0&0\\
0&1&0&0&0&0&0\\
0&0&0&0&0&0&0\\
0&0&0&1&0&0&0\\
0&0&1&0&0&0&0\\
1&0&0&0&0&0&0\\
\end{array}\right]\right).$
\end{ex}

\begin{theorem}\label{number.bell}
The total number of $GL_2(T_n)$-orbits is equal to the Bell number $B_n$.
\end{theorem}
\begin{proof} Recall that the Bell number $B_n$ is  the number  of partitions $\Pi$ of the set $\{1, 2, \dots n\}$. Therefore,
in the light of Corollary \ref{number.pairs} it suffices to show that the set of pairs ${\tt(A, B)}\in{^2T_n}$ and the family  of partitions $\Pi$ are equipotent.

\noindent
Let $A_{k}$ be the matrix obtained from ${\tt A}=\left(a_{kl}\right)$ by removing its rows and columns number $k+1, k+2, \dots, n$ if $k=1, 2, \dots, n-1$, and let $A_{k}={\tt A}$ if $k=n$. The rank of a matrix $A_{k}$ is denoted by $t_k$ and the expression $T$ means the set $\left\{t_k; k=1, 2, \dots, n\right\}$.
There exists a bijection that associates to every pair ${\tt(A, B)}\in{^2T_n}$ the partition $\Pi=\{U_{t_k}, t_k\in T\}$ such that for any $k\in \{1, 2, \dots n\}$ with $a_{kk}=1$
$$U_{t_k}=\left\{k\right\}\cup\left\{i; b_{ij}=1,
t_k=j-rankB_{ij}, 1\leqslant i\leqslant n, 1\leqslant j\leqslant n\right\},$$
where $B_{ij}$ is the matrix obtained from ${\tt B}=\left(b_{ij}\right)$ by changing its rows number $i, i+1, \dots, n$ and columns number $j, j+1, \dots, n$ into zero rows and columns.

Conversely, to a partition $\Pi=\{U_s, s\in S\}$ of the set $\{1, 2, \dots n\}$ corresponds  exactly one pair ${\tt(A, B)}\in{^2T_n}$. To see this correspondence,  suppose that  $u_s$ is the smallest element of the set $U_s$, where $s\in S$.
An entry $a_{ij}$ of $\tt{A}=(a_{ij})$ is equal to $1$  if $i=j=u_s$ for some $s\in S$, being equal to
$0$ otherwise.

\noindent
Elements of the set $$\{1, 2, \dots n\}\backslash \{u_s, s\in S\}=\{x_m; m=1, 2, \dots, n- rank{\tt A}\},$$ where $x_1< x_2< \dots< x_{n- rank{\tt A}}$, are numbers of rows with entry $1$ of the matrix ${\tt B}$. Any other row of ${\tt B}$ is a zero row.
We have to indicate the column number $y_m$ containing entry $1$ from a row $x_m$ for any $m=1, 2, \dots, n- rank{\tt A}$.
We do it for each row one by one, from the smallest $x_m$, i.e. $x_1$, to the greatest, i.e. $x_{n- rank{\tt A}}$. This order is crucial.
For any $m=1, 2, \dots, n- rank{\tt A}$ the column $y_m$ is the $s$-th column without entry $1$ (from the left). Thereby, entry $1$ in the $x_1$-th row is in the $s$-th column, where $x_1\in U_s$.
\end{proof}
\begin{cor}
The orbit of unimodular cyclic submodules of ${^2T_n}$ represented by the pair $(1, 0)$ of ${^2T_n}$ corresponds to the partition $\Pi=\left\{\left\{1\right\}, \left\{2\right\}, \dots, \left\{n\right\}\right\}$  of the set $\{1, 2, \dots n\}$.
\end{cor}
\begin{ex}
Choose the pair
$${\tt(A, B)}=\left(\left[\begin{array}{cccclr}
1&0&0&0&0&0\\ 0&1&0&0&0&0\\0&0&0&0&0&0\\0&0&0&1&0&0\\0&0&0&0&0&0\\0&0&0&0&0&0\end{array}\right],\left[\begin{array}{cccclr}
0&0&0&0&0&0\\ 0&0&0&0&0&0\\0&1&0&0&0&0\\0&0&0&0&0&0\\0&0&0&1&0&0\\0&0&1&0&0&0\end{array}\right]\right)\in{^2T_6}.$$ We have $a_{kk}=1$ for $k=1, 2, 4$, so $\Pi=\{U_{t_1}, U_{t_2}, U_{t_4}\}$, where $t_1=1$, $t_2=2$, $t_4=3$. Of course, $1\in U_1$, $2\in U_2$, $4\in U_3$. Moreover,
\begin{itemize}
\item $b_{32}=1$, thus $j-rankB_{32}=2-0=2$ and hence $3\in U_2$;
\item$b_{54}=1$, thus $j-rankB_{54}=4-1=3$ and hence $5\in U_3$;
\item$b_{63}=1$, thus $j-rankB_{63}=3-1=2$ and hence $6\in U_2$.\end{itemize}
Therefore,  the partition $\Pi=\left\{\{1\}, \{2,3,6\}, \{4, 5\}\right\}$  corresponds to such a pair ${\tt(A, B)}$.

\noindent
We will also show the converse correspondence.
As $\{u_s, s\in S\}=\{1, 2, 4\}$, so $a_{11}=a_{22}=a_{44}=1$ and all the remaining entries of $\tt{A}$ are zero. Moreover, $x_1=3, x_2=5, x_3=6$, hence rows $3$, $5$ and $6$ of $\tt{B}$ contain entry $1$. $3\in U_2$, thus $b_{32}=1$. $5\in U_3$ and the $3$rd column without entry $1$ (from the left) is now the column number $4$, therefore $b_{54}=1$. By the same way we get that $b_{63}=1$, all the remaining entries of $\tt{B}$ being zero. Finally, the pair of ${^2T_6}$ corresponding to $\Pi$ is equal to ${\tt(A, B)}$.
\end{ex}

\begin{ex}
Consider the ring $T_4$. By Theorem \ref{number.bell} there are exactly $B_4=15$ $GL_2(T_4)$-orbits with the following representatives and corresponding partitions of the set $\{1, 2, 3, 4\}$:

\vspace*{0.3cm}

$T_4\left(\left[\begin{array}{cccclr}
1&0&0&0\\ 0&1&0&0\\0&0&1&0\\0&0&0&1\end{array}\right],\left[\begin{array}{cccclr}
0&0&0&0\\ 0&0&0&0\\0&0&0&0\\0&0&0&0\end{array}\right]\right), \ \Pi=\{\{1\}, \{2\}, \{3\}, \{4\}\};$

\vspace*{0.3cm}

$T_4\left(\left[\begin{array}{cccclr}
1&0&0&0\\ 0&1&0&0\\0&0&1&0\\0&0&0&0\end{array}\right],\left[\begin{array}{cccclr}
0&0&0&0\\ 0&0&0&0\\0&0&0&0\\1&0&0&0\end{array}\right]\right), \ \Pi=\{\{1,4\}, \{2\}, \{3\}\};$

\vspace*{0.3cm}

$T_4\left(\left[\begin{array}{cccclr}
1&0&0&0\\ 0&1&0&0\\0&0&1&0\\0&0&0&0\end{array}\right],\left[\begin{array}{cccclr}
0&0&0&0\\ 0&0&0&0\\0&0&0&0\\0&1&0&0\end{array}\right]\right), \ \Pi=\{\{1\}, \{2, 4\}, \{3\}\};$

\vspace*{0.3cm}

$T_4\left(\left[\begin{array}{cccclr}
1&0&0&0\\ 0&1&0&0\\0&0&1&0\\0&0&0&0\end{array}\right],\left[\begin{array}{cccclr}
0&0&0&0\\ 0&0&0&0\\0&0&0&0\\0&0&1&0\end{array}\right]\right), \ \Pi=\{\{1\}, \{2\}, \{3, 4\}\};$

\vspace*{0.3cm}

$T_4\left(\left[\begin{array}{cccclr}
1&0&0&0\\ 0&1&0&0\\0&0&0&0\\0&0&0&1\end{array}\right],\left[\begin{array}{cccclr}
0&0&0&0\\ 0&0&0&0\\1&0&0&0\\0&0&0&0\end{array}\right]\right), \ \Pi=\{\{1,3\}, \{2\}, \{4\}\};$

\vspace*{0.3cm}

$T_4\left(\left[\begin{array}{cccclr}
1&0&0&0\\ 0&1&0&0\\0&0&0&0\\0&0&0&1\end{array}\right],\left[\begin{array}{cccclr}
0&0&0&0\\ 0&0&0&0\\0&1&0&0\\0&0&0&0\end{array}\right]\right), \ \Pi=\{\{1\}, \{2, 3\}, \{4\}\};$

\pagebreak

$T_4\left(\left[\begin{array}{cccclr}
1&0&0&0\\ 0&0&0&0\\0&0&1&0\\0&0&0&1\end{array}\right],\left[\begin{array}{cccclr}
0&0&0&0\\ 1&0&0&0\\0&0&0&0\\0&0&0&0\end{array}\right]\right), \ \Pi=\{\{1,2\}, \{3\}, \{4\}\};$

\vspace*{0.3cm}

$T_4\left(\left[\begin{array}{cccclr}
1&0&0&0\\ 0&1&0&0\\0&0&0&0\\0&0&0&0\end{array}\right],\left[\begin{array}{cccclr}
0&0&0&0\\ 0&0&0&0\\1&0&0&0\\0&1&0&0\end{array}\right]\right), \ \Pi=\{\{1,3,4\}, \{2\}\};$

\vspace*{0.3cm}

$T_4\left(\left[\begin{array}{cccclr}
1&0&0&0\\ 0&1&0&0\\0&0&0&0\\0&0&0&0\end{array}\right],\left[\begin{array}{cccclr}
0&0&0&0\\ 0&0&0&0\\1&0&0&0\\0&0&1&0\end{array}\right]\right), \ \Pi=\{\{1,3\}, \{2,4\}\};$

\vspace*{0.3cm}

$T_4\left(\left[\begin{array}{cccclr}
1&0&0&0\\ 0&1&0&0\\0&0&0&0\\0&0&0&0\end{array}\right],\left[\begin{array}{cccclr}
0&0&0&0\\ 0&0&0&0\\0&1&0&0\\1&0&0&0\end{array}\right]\right), \ \Pi=\{\{1,4\}, \{2,3\}\};$

\vspace*{0.3cm}

$T_4\left(\left[\begin{array}{cccclr}
1&0&0&0\\ 0&1&0&0\\0&0&0&0\\0&0&0&0\end{array}\right],\left[\begin{array}{cccclr}
0&0&0&0\\ 0&0&0&0\\0&1&0&0\\0&0&1&0\end{array}\right]\right), \ \Pi=\{\{1\}, \{2,3,4\}\};$

\vspace*{0.3cm}

$T_4\left(\left[\begin{array}{cccclr}
1&0&0&0\\ 0&0&0&0\\0&0&1&0\\0&0&0&0\end{array}\right],\left[\begin{array}{cccclr}
0&0&0&0\\ 1&0&0&0\\0&0&0&0\\0&1&0&0\end{array}\right]\right), \ \Pi=\{\{1,2,4\}, \{3\}\};$

\vspace*{0.3cm}

$T_4\left(\left[\begin{array}{cccclr}
1&0&0&0\\ 0&0&0&0\\0&0&1&0\\0&0&0&0\end{array}\right],\left[\begin{array}{cccclr}
0&0&0&0\\ 1&0&0&0\\0&0&0&0\\0&0&1&0\end{array}\right]\right), \ \Pi=\{\{1,2\}, \{3, 4\}\};$

\vspace*{0.3cm}

$T_4\left(\left[\begin{array}{cccclr}
1&0&0&0\\ 0&0&0&0\\0&0&0&0\\0&0&0&1\end{array}\right],\left[\begin{array}{cccclr}
0&0&0&0\\ 1&0&0&0\\0&1&0&0\\0&0&0&0\end{array}\right]\right), \ \Pi=\{\{1,2,3\}, \{4\}\};$

\pagebreak

$T_4\left(\left[\begin{array}{cccclr}
1&0&0&0\\ 0&0&0&0\\0&0&0&0\\0&0&0&0\end{array}\right],\left[\begin{array}{cccclr}
0&0&0&0\\ 1&0&0&0\\0&1&0&0\\0&0&1&0\end{array}\right]\right), \ \Pi=\{\{1,2,3,4\}\}.$
\end{ex}

\end{document}